\documentclass[16pt, oneside, reqno]{amsart}
\usepackage{amsmath,amssymb,amsthm,enumerate,color,graphics,graphicx,epsfig,url}
\usepackage[T2A]{fontenc}
\usepackage[utf8]{inputenc}
\usepackage[english]{babel}
\usepackage{xcolor}
\usepackage{latexsym,amsfonts,mathrsfs}
\usepackage{hhline}
\usepackage{euscript}
\usepackage{cite,cmap}
\usepackage{subcaption}
\usepackage{textcase}
\usepackage{amsaddr}
\usepackage{hyperref}

\newtheorem{theorem}{Theorem}[section]
\newtheorem{lemma}[theorem]{Lemma}
\newtheorem{proposition}[theorem]{Proposition}
\newtheorem{corollary}[theorem]{Corollary}
\newtheorem{definition}[theorem]{Definition}

\theoremstyle{definition}
\newtheorem{example}[theorem]{Example}

\makeatletter
\renewenvironment{abstract}
  {\ifx\maketitle\relax
     \ClassWarning{\@classname}{Abstract should precede \protect\maketitle
       \space in AMS document classes; reported}%
   \fi
   \global\setbox\abstractbox=\vtop\bgroup
     \normalfont\Small
     \list{}{\labelwidth\z@
       \leftmargin3pc
       \rightmargin\leftmargin
     \itemindent\z@
     \listparindent\normalparindent
     \parsep\z@ \@plus\p@
     }
     \item[]
  }
  {\endlist\egroup
   \ifx\@setabstract\relax \@setabstracta \fi
  }
\makeatother

\makeatletter
\def\@settitle{\begin{flushleft}\Large\bfseries\@title\end{flushleft}}
\makeatother

\title[H\"older exponents and level sets of self-affine functions]{H\"older exponents and fractal structure of level sets of self-affine\\ functions associated with the $Q_s$-representation of numbers}

\date{}

\begin{document}



\maketitle

\noindent\textbf{Volodymyr Yelahin}

\noindent\textit{Institute of Mathematics, National Academy of Sciences of Ukraine \\
Tereshchenkivska Str., 3, Kyiv, 01024, Ukraine\\
e-mail: yelahin@imath.kiev.ua\\
https://orcid.org/0009-0001-6984-5090}

~

\noindent\textbf{Mykola Moroz}

\noindent\textit{Institute of Mathematics, National Academy of Sciences of Ukraine \\
Tereshchenkivska Str., 3, Kyiv, 01024, Ukraine\\
e-mail: mykola.moroz@imath.kiev.ua, corresponding author\\
https://orcid.org/0000-0001-6658-4924\\
[6pt]
Dragomanov Ukrainian State University\\
Pyrohova str., 9, Kyiv, 01054, Ukraine}

\begin{abstract}
We investigate a class of locally complicated self-affine functions defined via the $Q_s$-representation of real numbers. In particular, we compute local Hölder exponents at points with given asymptotic frequencies of digits in their $Q_s$-representation. Furthermore, we establish conditions under which these functions possess continuum level sets. Finally, for self-affine functions satisfying additional conditions, we describe the geometric structure of the set of maximum points and show that this set can be fractal.  
\end{abstract}

local H\"older exponent, self-similar set, fractal set.


\section{Introduction}

Among functions with complex local structure, self-affine functions occupy a prominent place. Historically, they became one of the simplest models within which examples of continuous nowhere monotonic, nowhere differentiable, and singular functions were constructed (Salem and Cantor singular functions, Takagi and Bolzano nowhere differentiable functions; see \cite{Dubuc2018, Jarn}). The study of self-affine functions played a significant role in the development of modern fractal geometry and fractal analysis, outlining new approaches to the study of irregular mathematical objects. This research area remains relevant today in the context of analyzing the metric and fractal properties of such functions.

According to \cite{Bedford1989,Dubuc2018,SZ1999}, the function $F\colon [a,b]\to\mathbb{R}$ is self-affine if its graph $G_F$ is the union of $s\geq2$ sets, each of which is the image of $G_F$ under some affine transformation $T_i$, i.e.,
\begin{equation*}
G_F=\bigcup_{i=1}^s T_i(G_F).
\end{equation*}
Traditionally, for transformations $T_i$ to correctly define a unique continuous function $F$ on the interval $[a,b]$, additional conditions are imposed:
\begin{enumerate}
\item[1.] $T_i((x,y))=(a_ix+b_i, c_ix+d_iy+e_i)$ and $|d_i|<1$ for all $i=1,\ldots,s$;
\item[2.] there exists a sequence of plane points $(x_0,y_0),(x_1,y_1),\ldots,(x_s,y_s)$ such that $a=x_0<x_1<\cdots<x_s=b$ and
\begin{equation*}
T_i\{(x_0,y_0),(x_s,y_s)\}=\{(x_{i-1},y_{i-1}),(x_i,y_i)\}
\end{equation*}
for all $i=1,\ldots,s$.
\end{enumerate}
The continuous function $F$ defined by such transformations $T_i$ for all $i=1,\ldots,s$ satisfies the equality
\begin{equation}
F(a_ix+b_i)=d_iF(x)+c_ix+e_i.
\end{equation}
If all parameters $a_i>0$, then $F$ is called \cite{Barnsley1988, Dubuc2018} a fractal interpolation function.

This work is devoted to self-affine functions defined in terms of the $Q_s$-representation of numbers.

Let $s\geq 2$ be a fixed natural number, $A_s=\{0,1,\ldots,s-1\}$, and $Q_s=\{q_0,\ldots,q_{s-1}\}$ be a positive stochastic vector. Define the numbers $\beta_0,\ldots,\beta_{s-1}$ as follows:
\begin{equation*}
\beta_0=0,\quad \beta_i=q_0+\cdots+q_{i-1} \text{ ~for all } i\in A_s\setminus\{0\}.
\end{equation*}
It is known \cite{AKPrT2011, Pr1998} that for every $x\in[0,1]$ there exists an infinite sequence $(\alpha_n)_{n=1}^\infty\in A_s^\mathbb{N}$ such that
\begin{equation}\label{Q-representation}
x=\beta_{\alpha_1}+\sum_{k=2}^\infty\left(\beta_{\alpha_k}\prod_{j=1}^{k-1}q_{\alpha_j}\right).
\end{equation}
The series \eqref{Q-representation} and its sum $x$ are briefly denoted by $\Delta^{Q_s}_{\alpha_1 \alpha_2\ldots}$. The expansion of the number $x$ into the series \eqref{Q-representation} is called the \emph{$Q_s$-representation of the number $x$}, and the numbers $\alpha_i = \alpha_i(x)$ are called the \emph{$Q_s$-digits} of this $Q_s$-representation.

Let the numbers $g_0,\ldots,g_{s-1}$ and $\delta_0,\ldots,\delta_{s-1}$ satisfy the conditions
\begin{gather}\label{conditions2}
\begin{cases}
0<|g_i|<1 \text{ for all } i\in A_s, \\
g_0+g_1+\cdots+g_{s-1}=1,\\
\delta_0=0,\quad \delta_i=g_0+\cdots+g_{i-1} \text{ for all } i\in A_s\setminus\{0\}.
\end{cases}
\end{gather}
The main object of our study is the function $f$ defined by
\begin{equation}\label{mainfunction}
f(x)=f\left(\Delta^{Q_s}_{\alpha_1 \alpha_2\ldots}\right)=\delta_{\alpha_1}+\sum_{k=2}^\infty\left(\delta_{\alpha_k}\prod_{j=1}^{k-1}g_{\alpha_j}\right).
\end{equation}
The function $f$ is a fractal interpolation function and is a solution to the system of functional equations
\begin{equation}
F\left(\beta_i+q_i x\right)=\delta_i+g_i F(x),\quad i\in A_s.
\end{equation}

In works \cite{Barnsley1988, Bedford1989, Dubuc2018, PrK2013, SZ1999}, conditions for nowhere monotonicity, singularity, and nowhere differentiability of the function $f$ were studied, the Hausdorff dimension of its graph, its global H\"older exponent and the local H\"older exponent for almost all points (in the sense of Lebesgue measure) were calculated. In \cite{PrGLSv2022}, the functions defined by equality \eqref{mainfunction} were generalized to the case of the $Q^*_s$-representation of numbers, which allowed the elimination of self-affinity and the acquisition of new fractal effects.

In this paper, we continue the mentioned research and calculate the local H\"older exponent of the function $f$ at points that have fixed frequencies of digits in their $Q_s$-representation. Special attention is paid to the case where $g_{k}<0$ and $\delta_k>1$ for some $k\in A_s$ and $g_i>0$ for all $i\in A_s\setminus\{k\}$, as a result of which the function $f$ is nowhere monotonic. In this case, its extrema are calculated, conditions under which there exist continuum level sets are established, and the set of maximum points has a Cantor-type structure; the non-invariance of the Hausdorff dimension, Baire category, and Lebesgue null-measure under the transformation $f$ is proved.

\section{Preliminaries}

\begin{definition}[{\cite{AKPrT2011,Pr1998}}]
A $Q_s$-representation $\Delta^{Q_s}_{\alpha_1 \alpha_2\ldots}$ with $\alpha_n=i$ for all $n>k$ is called periodic with period $(i)$ and is denoted by $\Delta^{Q_s}_{\alpha_1 \ldots\alpha_k (i)}$. 
\end{definition}

\begin{proposition}[\cite{AKPrT2011,Pr1998}, Main theorem on $Q_s$-representation]
  For any positive stochastic vector $Q_s$, every number $x\in[0,1]$ admits at most two $Q_s$-representations. Moreover, if $x$ admits two $Q_s$-representations, then these representations are exactly of the form  $\Delta^{Q_s}_{\alpha_1\ldots\alpha_n (0)}$ and $\Delta^{Q_s}_{\alpha_1\ldots\alpha_{n-1}[\alpha_n-1] (s-1)}$, where $\alpha_n\not=0$.
\end{proposition}

\begin{definition}[\cite{AKPrT2011,Pr1998}]
Numbers that have a unique $Q_s$-representation are called $Q_s$-unary, and numbers that have two $Q_s$-representations are called $Q_s$-binary.     
\end{definition}

\begin{proposition}[\cite{AKPrT2011,Pr1998}]
    If $x_1 = \Delta^{Q_s}_{\alpha_1\alpha_2\dots}$ and $x_2 = \Delta^{Q_s}_{\gamma_1\gamma_2\dots}$ are two distinct numbers, then $x_1 < x_2$  if and only if there exists $k \in \mathbb{N}$ such that $\alpha_i = \gamma_i$ for all $i<k$ and $\alpha_{k} < \gamma_{k}$.
\end{proposition}

\begin{definition}
Let $\left(c_1, \ldots, c_k\right)$ be a fixed sequence of digits from $A_s$. The set $\Delta^{Q_s}_{c_1\ldots c_k}$ of all numbers in $[0,1]$ that admit a $Q_s$-representation of the form $\Delta^{Q_s}_{c_1\ldots c_k\alpha_{k+1}\ldots}$ is called \cite{Pr1998} a $Q_s$-cylinder of rank $k$ with base $c_1\ldots c_k$. 
\end{definition}

\begin{proposition}
The $Q_s$-cylinders have the following properties \cite{PrK2013}: 
\end{proposition} 
\begin{enumerate}
\item[1.] $\Delta^{Q_s}_{c_1\ldots c_k}=[a,b]$, where $a=\Delta^{Q_s}_{c_1\ldots c_k(0)}$ and $b=\Delta^{Q_s}_{c_1\ldots c_k(s-1)}$;
\item[2.] $[0,1]=\bigcup\limits_{t=0}^{s-1}\Delta^{Q_s}_{t},\quad\Delta^{Q_s}_{c_1\ldots c_k}=\bigcup\limits_{t=0}^{s-1}\Delta^{Q_s}_{c_1\ldots c_k t}$;
\item[3.]
$\max \Delta^{Q_s}_{c_1\ldots c_k i} = \min\Delta^{Q_s}_{c_1\ldots c_k (i+1)}, \quad i\neq s-1;$
\item[4.] 
$\left|\Delta^{Q_s}_{c_1\ldots c_k}\right|=\prod\limits_{i=1}^k q_{c_i}\;\to\;0  \text{ ~as }k\to\infty;$
\item[5.]
$\bigcap\limits_{k=1}^{\infty}\Delta^{Q_s}_{c_1\ldots c_k}=x=\Delta^{Q_s}_{c_1\ldots c_k\ldots } \text{ ~for any sequence } (c_k)_{k=1}^\infty\in A_s^{\mathbb{N}}.$
\end{enumerate}

\begin{definition}[\cite{APrT2005,Pr1998}]
    Let $N_i(x, n)$ denote the number of occurrences of the digit $i \in A_s$ among the first $n$ digits in the $Q_s$-representation of $x$. If the limit
    \begin{equation*}
    \lim\limits_{n\to\infty} \frac{N_i(x, n)}{n} = \nu_i(x),
    \end{equation*}
exists, then it is called the frequency of the digit $i$ in the $Q_s$-representation of $x$.
\end{definition}

For $Q_s$-unary numbers, this definition is correctly defined. The frequency of digits of a $Q_s$-binary number $x$ will be correctly defined only after we agree which of its two representations we will use.

\begin{proposition}[\cite{APrT2005,Pr1998}]\label{frequency}
    For Lebesgue almost all $x\in[0,1]$, 
    \begin{equation*}
        \nu_0(x)=q_0,~\ldots,~\nu_{s-1}(x)=q_{s-1}.
    \end{equation*}
\end{proposition}

In \cite{PrK2013}, the main properties of the function $f$ were established:
\begin{enumerate}
\item[1.] the function $f$ is well-defined, i.e., the series \eqref{mainfunction} converges, and its value at each $Q_s$-binary point does not depend on the chosen $Q_s$-representation of such a point;
\item[2.] the function $f$ is continuous;
\item[3.] if $g_i> 0$ for all $i\in A_s$ and $g_k\neq q_k$ for some $k\in A_s$, then $f$ is an increasing singular function;
\item[4.] if $g_k<0$ for some $k\in A_s$, then $f$ is nowhere monotonic;
\item[5.] if $|g_i|>q_i$ for all $i\in A_s$, then $f$ is a nowhere differentiable function.
\end{enumerate}

\begin{proposition}
For the function $f$, the following equalities hold:
\begin{gather}
    f(0)=f\left(\Delta^{Q_s}_{(0)}\right)=0,\quad f(1)=f\left(\Delta^{Q_s}_{(s-1)}\right)=1;\label{f01}\\
    f\left(\Delta^{Q_s}_{\alpha_1\ldots\alpha_n\alpha_{n+1}\ldots}\right)=\delta_{\alpha_1}+\sum\limits_{k=2}^{n}\left(\delta_{\alpha_k}\prod\limits_{j=1}^{k-1}g_{\alpha_j}\right)+\prod\limits_{j=1}^{n}g_{\alpha_j}\cdot f\left(\Delta^{Q_s}_{\alpha_{n+1}\alpha_{n+2}\ldots}\right);\label{easy1}\\
    f\left(\Delta^{Q_s}_{c_1\ldots c_n \alpha_{n+1}\ldots}\right)-f\left(\Delta^{Q_s}_{c_1\ldots c_n \beta_{n+1}\ldots}\right)=\prod\limits_{j=1}^{n}g_{c_j}\cdot \left(f\left(\Delta^{Q_s}_{\alpha_{n+1}\alpha_{n+2}\ldots}\right)-f\left(\Delta^{Q_s}_{\beta_{n+1}\beta_{n+2}\ldots}\right)\right).\label{easy2}
\end{gather}   
\end{proposition}

\begin{proposition}
If $g_i<0$ for some $i\in A_s$, then $f$ has an unbounded variation.
\end{proposition}

\begin{proof}
We estimate the variation $V_0^1(f)$ from below. Consider a partition of the interval $[0,1]$ by $Q_s$-cylinders of rank $n$ and define $V_n$ given by
\begin{equation*}
V_n =\sum_{(c_1,\ldots, c_n)\in A_s^n}\left|f\left(\max\Delta^{Q_s}_{c_1\ldots c_n}\right)-f\left(\min\Delta^{Q_s}_{c_1\ldots c_n}\right)\right|.
\end{equation*}
Clearly, $V_0^1(f) \geq V_n$ for any $n \in \mathbb{N}$. Let us calculate the value of $V_n$ using equality \eqref{easy2}: 

\begin{align*}
V_n&= \sum_{(c_1,\ldots, c_n)\in A_s^n}\left|f\left(\Delta_{c_1\ldots c_n(s-1)}^{Q_s}\right)-f\left(\Delta_{c_1\ldots c_n(0)}^{Q_s}\right)\right|=\\
&=\sum_{(c_1,\ldots, c_n)\in A_s^n} \left(\prod_{j=1}^{n}\left|g_{c_j}\right|\cdot\left|f\left(\Delta^{Q_s}_{(s-1)}\right)-f\left(\Delta^{Q_s}_{(0)}\right)\right|\right) = \\
&=\sum_{(c_1,\ldots, c_n)\in A_s^n} \prod_{j=1}^{n}\left|g_{c_j}\right| = \left(\left|g_0\right|+\cdots+\left|g_{s-1}\right|\right)^n.
\end{align*}
From the condition of the theorem we have
$\left|g_0\right|+\cdots+\left|g_{s-1}\right|>1$, and therefore $\lim\limits_{n \to \infty} V_n=\infty$. Hence, $f$ has unbounded variation.
\end{proof}

Since the derivative does not allow for a complete description of the local structure of singular and nowhere differentiable functions, the local and global H\"older exponents are used for their analysis

\begin{definition}[{\cite{SZ1999}}]
A function $F\colon I\to\mathbb{R}$ is said to satisfy the H\"older condition (is H\"older continuous) with the exponent $\alpha>0$ on an interval $I\subseteq\mathbb{R}$ if
\begin{equation*}
    \sup_{x,y\in I,~x\not=y}\frac{\left|F(x)-F(y)\right|}{|x-y|^\alpha}<\infty.
\end{equation*}    
\end{definition}

If a function $F$ satisfies the H\"older condition with the exponent $\alpha$, then it also satisfies the H\"older condition with every exponent $\alpha'\in(0,\alpha]$. The number
\begin{equation*}
\alpha_F=\sup\left\{\alpha>0\colon \sup_{x,y\in I,~x\not=y}\frac{\left|F(x)-F(y)\right|}{|x-y|^\alpha}<\infty\right\}
\end{equation*}
is called \cite{SZ1999} the \emph{global H\"older exponent} of the function $F$ on the interval $I$. If the function $F$ does not satisfy the~H\"older condition for any $\alpha>0$, then by convention $\alpha_F=0$. 

\begin{proposition}[\cite{Dubuc2018}]
The function $f$ satisfies the H\"older condition with exponents
\begin{equation*}
   \alpha\leq\min\limits_{i\in A_s} \frac{\ln |g_i|}{\ln q_i}.
\end{equation*}
\end{proposition}

\begin{definition}[{\cite{SZ1999}}]
A function $F$ is said to satisfy the H\"older condition (is H\"older continuous) with the exponent $\alpha>0$ at a point $x_0$ if
\begin{equation*}
    \limsup_{x\to x_0}\frac{\left|F(x_0)-F(x)\right|}{\left|x_0-x\right|^\alpha}<\infty.
\end{equation*}
\end{definition}

The \emph{local H\"older exponent} of a function $F$ at a point $x_0$ is the number \cite{Bedford1989,SZ1999}
\begin{equation*}
\alpha_F(x_0)=\sup\left\{\alpha>0\colon \limsup_{x\to x_0}\frac{\left|F(x_0)-F(x)\right|}{\left|x_0-x\right|^\alpha}<\infty\right\}=\liminf_{x\to x_0}\frac{\ln \left|F(x_0)-F(x)\right|}{\ln \left|x_0-x\right|}.
\end{equation*}
Similarly, if the function $F$ does not satisfy the H\"older condition at the point $x_0$ for any $\alpha>0$, then by convention $\alpha_F(x_0)=0$.

\begin{proposition}[\cite{Dubuc2018}]
For Lebesgue almost all $x\in[0,1]$, the local H\"older exponent of the function $f$ at the point $x$ equals 
\begin{equation*}
    \alpha_f(x)=\frac{\sum\limits_{i=0}^{s-1} q_i\ln|g_i|}{\sum\limits_{i=0}^{s-1} q_i\ln q_i}.
\end{equation*}
\end{proposition}

\section{Local H\"older Exponents of the Function $f$}

\begin{lemma}\label{ti}
Let $x=\Delta^{Q_s}_{\alpha_1 \alpha_2\ldots}$ be a $Q_s$-unary number with digit frequencies $\nu_i(x)<1$ for all $i\in A_s$, and let 
\begin{equation*}
 t_i(x,n)=\min\left\{t\geq 0\colon \alpha_{n+t+1}\neq i\right\}.  
\end{equation*}
Then 
\begin{equation*}
    \lim_{n\to\infty}\frac{t_i(x,n)}{n}=0.
\end{equation*}
\end{lemma}

\begin{proof}
By assumption $\lim\limits_{n\to\infty}\frac{N_i(x,n)}{n}=\nu_i(x)<1$. Then
\begin{equation*}
    \lim_{n\to\infty}\frac{N_i(x,n)+t_i(x,n)}{n+t_i(x,n)}=\lim_{n\to\infty}\frac{N_i(x,n+t_i(x,n))}{n+t_i(x,n)}=\nu_i(x),
\end{equation*}
since the sequence $\left\{\frac{N_i(x,n+t_i(x,n))}{n+t_i(x,n)}\right\}_{n=1}^\infty$ is a subsequence of $\left\{\frac{N_i(x,n)}{n}\right\}_{n=1}^\infty$. Consider the difference between these expressions:
\begin{equation*}
\frac{N_i(x,n)+t_i(x,n)}{n+t_i(x,n)}-\frac{N_i(x,n)}{n}=\frac{t_i(x,n)\cdot(n-N_i(x,n))}{n\cdot(n+t_i(x,n))}=\frac{\frac{t_i(x,n)}{n}\left(1-\frac{N_i(x,n)}{n}\right)}{1+\frac{t_i(x,n)}{n}}.
\end{equation*}
Hence,
\begin{equation*}
    \lim_{n\to\infty}\frac{\frac{t_i(x,n)}{n}}{1+\frac{t_i(x,n)}{n}}=0
\end{equation*}
and $\lim\limits_{n\to\infty}\frac{t_i(x,n)}{n}=0$.
\end{proof}

Denote by $L$ the quantity
\begin{equation*}
 L=\frac{\sum\limits_{i=0}^{s-1} \nu_i\ln|g_i|}{\sum\limits_{i=0}^{s-1} \nu_i\ln q_i}.
\end{equation*}

\begin{theorem}\label{9}
Let $x_0$ be a $Q_s$-unary number whose $Q_s$-representation has digit frequencies $\nu_i$, $i \in A_s$, with $\nu_0,~\nu_{s-1}<1$. Then the function $f$ satisfies the H\"older condition with exponent $\alpha<L$ at $x_0$.
\end{theorem}
 
\begin{proof}
Let $n$ be the number of first common digits in the $Q_s$-representations of $x_0$ and $x$, i.e., $x_0 = \Delta^{Q_s}_{b_1\dots b_n b_{n+1}\dots}$,
$x=\Delta^{Q_s}_{b_1\dots b_n c_{n+1}\dots}$, where $b_{n+1}\neq c_{n+1}$. For a $Q_s$-unary point $x_0$, $x\to x_0$ is equivalent to $n \to \infty$. Then
    \begin{equation*}
        \frac{|f(x_0)-f(x)|}{|x_0-x|^\alpha}=\prod^n_{i=1} \frac{|g_{b_i}|}{q_{b_i}^\alpha} \cdot \frac{\left|f\left(\Delta^{Q_s}_{b_{n+1}\dots}\right)-f\left(\Delta^{Q_s}_{c_{n+1}\dots}\right)\right|}{\left|\Delta^{Q_s}_{b_{n+1}\dots}-\Delta^{Q_s}_{c_{n+1}\dots}\right|^\alpha}.
    \end{equation*}
Denote $N_i(x_0,n)$ by $N_i(n)$. Thus, we obtain
\begin{equation*}
\prod^n_{i=1} \frac{|g_{b_i}|}{q_{b_i}^\alpha}=\left(\frac{|g_0|}{q_0^\alpha}\right)^{N_0(n)}\cdots \left(\frac{|g_{s-1}|}{q_{s-1}^\alpha}\right)^{N_{s-1}(n)}=\left(\left(\frac{|g_0|}{q_0^\alpha}\right)^\frac{N_0(n)}{n}\cdots \left(\frac{|g_{s-1}|}{q_{s-1}^\alpha}\right)^\frac{N_{s-1}(n)}{n}\right)^n.
\end{equation*}
Moreover, 
\begin{gather*}
\frac{\left|f\left(\Delta^{Q_s}_{b_{n+1}\dots}\right)-f\left(\Delta^{Q_s}_{c_{n+1}\dots}\right)\right|}{\left|\Delta^{Q_s}_{b_{n+1}\dots}-\Delta^{Q_s}_{c_{n+1}\dots}\right|^{\alpha}} \leq \frac{M-m}{\min\left\{\Delta^{Q_s}_{b_{n+1}\ldots}-\Delta^{Q_s}_{b_{n+1}(0)},\quad \Delta^{Q_s}_{b_{n+1}(s-1)}-\Delta^{Q_s}_{b_{n+1}\ldots}\right\}},
\end{gather*} 
where $M=\max\limits_{x\in[0,1]}f(x)$ and $m=\min\limits_{x\in[0,1]}f(x)$. Denote $q_{\min}=\min\limits_{i\in A_s} q_i$ and $t_i(n)=t_i(x_0,n)$. Then
\begin{gather*}
\Delta^{Q_s}_{b_{n+1}\ldots}-\Delta^{Q_s}_{b_{n+1}(0)}\geq \Delta^{Q_s}_{b_{n+1}\!\!\!\!\!\underbrace{\scriptstyle 0\ldots 0}_{t_{0}(n+1)+1}\!\!\!\!\!(s-1)}-\Delta^{Q_s}_{b_{n+1}(0)}=q_{b_{n+1}}\cdot q_{0}^{t_{0}(n+1)+1}\geq q_{\min}^{t_{0}(n+1)+2},\\
\Delta^{Q_s}_{b_{n+1}(s-1)}-\Delta^{Q_s}_{b_{n+1}\ldots}\geq\Delta^{Q_s}_{b_{n+1}(s-1)}-\Delta^{Q_s}_{b_{n+1}\underbrace{\scriptstyle [s-1]\ldots [s-1]}_{t_{s-1}(n+1)+1}(0)}=q_{b_{n+1}}\cdot q_{s-1}^{t_{s-1}(n+1)+1}\geq q_{\min}^{t_{s-1}(n+1)+2}.
\end{gather*}
Hence
\begin{gather*}
\frac{\left|f\left(\Delta^{Q_s}_{b_{n+1}\dots}\right)-f\left(\Delta^{Q_s}_{c_{n+1}\dots}\right)\right|}{\left|\Delta^{Q_s}_{b_{n+1}\dots}-\Delta^{Q_s}_{c_{n+1}\dots}\right|^{\alpha}} \leq \frac{M-m}{q_{\min}^{\max\left\{t_{0}(n+1),~ t_{s-1}(n+1)\right\}+2}}\leq \frac{M-m}{q_{\min}^{t_{0}(n+1)+ t_{s-1}(n+1)+2}}.
\end{gather*}
Thus,
\begin{align*}
     \frac{|f(x_0)-f(x)|}{|x_0-x|^\alpha}\leq \left(\left(\frac{|g_0|}{q_0^\alpha}\right)^\frac{N_0(n)}{n}\cdots \left(\frac{|g_{s-1}|}{q_{s-1}^\alpha}\right)^\frac{N_{s-1}(n)}{n}\cdot\frac{1}{q_{\min}^\frac{t_{0}(n+1)+t_{s-1}(n+1)}{n}} \right)^n\cdot  \frac{M-m}{q_{\min}^{2}}.
\end{align*}
Since $\nu_0,~\nu_{s-1}<1$, by Lemma \ref{ti}
\begin{gather*}\label{ratioto0}
    \lim_{n\to\infty} \frac{t_0(n+1)+t_{s-1}(n+1)}{n}=0.
\end{gather*}
Using $\lim\limits_{n\to\infty}\frac{N_i(n)}{n}=\nu_i$, we obtain
\begin{equation*}
    \lim_{n\to\infty} \left(\left(\frac{|g_0|}{q_0^\alpha}\right)^\frac{N_0(n)}{n}\cdots \left(\frac{|g_{s-1}|}{q_{s-1}^\alpha}\right)^\frac{N_{s-1}(n)}{n}\cdot\frac{1}{q_{\min}^\frac{t_{0}(n+1)+t_{s-1}(n+1)}{n}}\right)=\left(\frac{|g_0|}{q_0^\alpha}\right)^{\nu_0} \cdots \left(\frac{|g_{s-1}|}{q_{s-1}^\alpha}\right)^{\nu_{s-1}}.
\end{equation*}
For $f$ to satisfy the H\"older condition at $x_0$ with exponent $\alpha$, it is sufficient that
\begin{equation*}
\left(\frac{|g_0|}{q_0^\alpha}\right)^{\nu_0} \cdots \left(\frac{|g_{s-1}|}{q_{s-1}^\alpha}\right)^{\nu_{s-1}}<1,  
\end{equation*}
which is equivalent to $\alpha<L$.
\end{proof}

\begin{theorem}
Let $x_0$ be a $Q_s$-unary number whose $Q_s$-representation has digit frequencies $\nu_i$, $i \in A_s$, with $\nu_0,~\nu_{s-1}<1$. Then the function $f$ does not satisfy the H\"older condition with the exponent $\alpha>L$ at $x_0$.  
\end{theorem}

\begin{proof}
Let $x_0=\Delta^{Q_s}_{b_1 b_2\ldots}$ be a fixed $Q_s$-unary number, and let $(x_n)^{\infty}_{n=1}$ be a sequence given by
\begin{equation*}
x_n=\Delta^{Q_s}_{b_1 \ldots b_n c^{(n)}_{1} c^{(n)}_{2}\ldots}, 
\end{equation*}
moreover, 
\begin{equation*}
   \left|f\left(\Delta^{Q_s}_{b_{n+1}b_{n+2}\ldots}\right)-f\left(\Delta^{Q_s}_{c^{(n)}_{1}c^{(n)}_{2}\ldots}\right)\right|=\frac{1}{2}.
\end{equation*}
This sequence $(x_n)^{\infty}_{n=1}$ exists because the range of the function $f$ is an interval that contains the interval $[0,1]$. Therefore, we obtain
\begin{equation*}
     \frac{|f(x_0)-f(x_n)|}{|x_0-x_n|^\alpha}=\prod^n_{i=1} \frac{|g_{b_i}|}{q_{b_i}^\alpha} \cdot \frac{\left|f\left(\Delta^{Q_s}_{b_{n+1}b_{n+2}\ldots}\right)-f\left(\Delta^{Q_s}_{c^{(n)}_{1}c^{(n)}_{2}\ldots}\right)\right|}{\left|\Delta^{Q_s}_{b_{n+1}b_{n+2}\ldots}-\Delta^{Q_s}_{c^{(n)}_{1}c^{(n)}_{2}\ldots}\right|^\alpha}\geq \frac{1}{2}\cdot \left(\left(\frac{|g_0|}{q_0^\alpha}\right)^\frac{N_0(n)}{n}\cdots \left(\frac{|g_{s-1}|}{q_{s-1}^\alpha}\right)^\frac{N_{s-1}(n)}{n}\right)^n
\end{equation*}
Under the condition $\alpha>L$,
\begin{equation*}
    \lim_{n\to\infty} \left(\left(\frac{|g_0|}{q_0^\alpha}\right)^\frac{N_0(n)}{n}\cdots \left(\frac{|g_{s-1}|}{q_{s-1}^\alpha}\right)^\frac{N_{s-1}(n)}{n}\right)=\left(\frac{|g_0|}{q_0^\alpha}\right)^{\nu_0}\cdots \left(\frac{|g_{s-1}|}{q_{s-1}^\alpha}\right)^{\nu_{s-1}}>1,
\end{equation*}
therefore, $\lim\limits_{n\to\infty}\frac{|f(x_0)-f(x_n)|}{|x_0-x_n|^\alpha}=\infty$ and $\limsup\limits_{x\to x_0}\frac{|f(x_0)-f(x)|}{|x_0-x|^\alpha}=\infty$.
\end{proof}

\begin{corollary}
Let $x$ be a $Q_s$-unary number whose $Q_s$-representation has digit frequencies $\nu_i$ for all $i \in A_s$ with $\nu_0,~\nu_{s-1}<1$. Then the local H\"older exponent $\alpha_f(x)$ of the function $f$ at the point $x$ is equal to $L$.
\end{corollary}

\begin{corollary}[\cite{Dubuc2018}]
For Lebesgue almost all $x\in[0,1]$, the local H\"older exponent of the function $f$ at the point $x$ is equal to 
\begin{equation*}
    \alpha_f(x)=\frac{\sum\limits_{i=0}^{s-1} q_i\ln|g_i|}{\sum\limits_{i=0}^{s-1} q_i\ln q_i}.
\end{equation*}
\end{corollary}
This corollary follows directly from Theorem \ref{9} and Proposition \ref{frequency}.

\begin{corollary}
If $\sum\limits_{i=0}^{s-1} q_i\ln|g_i|<\sum\limits_{i=0}^{s-1} q_i\ln q_i$, then $f$ is singular.
\end{corollary}

\begin{figure}[h]
	\centering
	\includegraphics[width=0.65\textwidth]{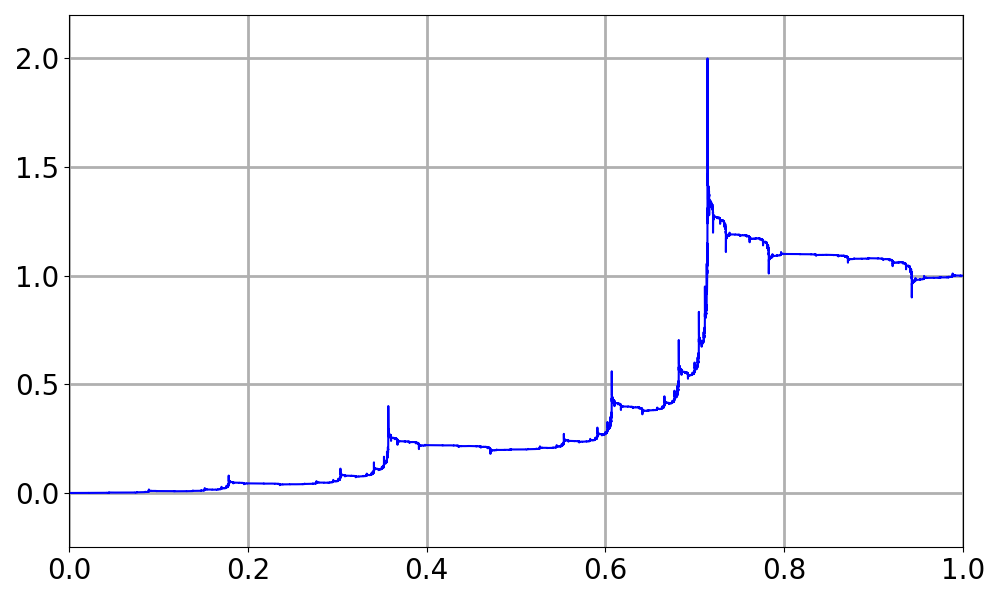}
	\caption{\centering The graph of the nowhere monotonic singular function $f$ for $q_0=\frac{1}{2}$, $q_1=\frac{3}{10}$, $q_2=\frac{1}{5}$ and $g_0=\frac{1}{5}$, $g_1=\frac{9}{10}$, $g_2=-\frac{1}{10}$.}
	\label{graphs1}
\end{figure}

Whether the H\"older condition with the exponent $\alpha = L$ holds at a $Q_s$-unary point $x$ depends on other properties of the sequences $\frac{N_i(x,n)}{n}$ to $\nu_i(x)$ and requires a more detailed analysis. However, this is beyond the scope of the~present study.

Denote by $K$ the quantity
\begin{equation*}
K=\min\left\{\frac{\ln |g_0|}{\ln q_0},\frac{\ln |g_{s-1}|}{\ln q_{s-1}}\right\}.
\end{equation*}

\begin{theorem}
The function $f$ satisfies the H\"older condition at a $Q_s$-binary point $x_0$ with exponent $\alpha\leq K$.
\end{theorem}

\begin{proof}
Since $x_0$ is a $Q_s$-binary number, we have $x_0 = \Delta^{Q_s}_{b_1 \ldots b_{k-1}b_k(0)}=\Delta^{Q_s}_{b_1\ldots b_{k-1}b_k-1
}$, where $b_k\neq 0$.

If $x>x_0$ and is sufficiently close to $x_0$, then
$x = \Delta^{Q_s}_{b_1\dots b_{k} \underbrace{\scriptstyle 0\ldots0}_{n-k}c_{n+1}\dots}$, where $c_{n+1} \neq 0$.
Here, $n$ is the number of first common digits in the $Q_s$-representation of $x$ and the $Q_s$-representation of $x_0$ with period $(0)$. In this case, the~condition $x\to x_0+$ is equivalent to $n\to \infty$. Analogously to the proof of Theorem \ref{9}, we have
\begin{equation*}
        \frac{|f(x_0)-f(x)|}{|x_0-x|^\alpha}=\left(\frac{|g_0|}{q_0^\alpha}\right)^{n-k}\cdot\prod^k_{i=1} \frac{|g_{b_i}|}{q_{b_i}^\alpha} \cdot \frac{\left|f\left(\Delta^{Q_s}_{(0)}\right)-f\left(\Delta^{Q_s}_{c_{n+1}\dots}\right)\right|}{\left|\Delta^{Q_s}_{(0)}-\Delta^{Q_s}_{c_{n+1}\dots}\right|^\alpha}.
\end{equation*}
Clearly, 
\begin{equation*}
\prod^k_{i=1} \frac{|g_{b_i}|}{q_{b_i}^\alpha}=\text{const},    
\end{equation*}
since $k$ is a constant uniquely determined by the number $x_0$. Also,
\begin{equation*}
 \frac{\left|f\left(\Delta^{Q_s}_{(0)}\right)-f\left(\Delta^{Q_s}_{c_{n+1}\dots}\right)\right|}{\left|\Delta^{Q_s}_{(0)}-\Delta^{Q_s}_{c_{n+1}\dots}\right|^\alpha}\leq \frac{M}{q_0^\alpha},
\end{equation*}
where $M=\max\limits_{x\in[0,1]}f(x)$. So, if $\alpha\leq \frac{\ln |g_0|}{\ln q_0}$, then $\frac{|g_0|}{q_0^\alpha}\leq 1$ and $\limsup\limits_{x\to x_0+}\frac{\left|f(x_0)-f(x)\right|}{\left|x_0-x\right|^\alpha}<\infty$.

If $x<x_0$ and is sufficiently close to $x_0$, then
$x = \Delta^{Q_s}_{b_1\dots [b_{k}-1] \underbrace{\scriptstyle [s-1]\ldots [s-1]}_{n-k}c_{n+1}\dots}$, where $c_{n+1} \neq s-1$.
Here, $n$ is the~number of first common digits in the $Q_s$-representation of $x$ and the $Q_s$-representation of $x_0$ with period $(s-1)$. In this case, the condition $x\to x_0-$ is also equivalent to $n\to \infty$. Then
\begin{equation*}
\frac{|f(x_0)-f(x)|}{|x_0-x|^\alpha}=\left(\frac{|g_{s-1}|}{q_{s-1}^\alpha}\right)^{n-k}\cdot\prod^k_{i=1} \frac{|g_{b_i}|}{q_{b_i}^\alpha} \cdot \frac{\left|f\left(\Delta^{Q_s}_{(s-1)}\right)-f\left(\Delta^{Q_s}_{c_{n+1}\dots}\right)\right|}{\left|\Delta^{Q_s}_{(s-1)}-\Delta^{Q_s}_{c_{n+1}\dots}\right|^\alpha}.
\end{equation*}
Clearly, 
\begin{equation*}
\prod^k_{i=1} \frac{|g_{b_i}|}{q_{b_i}^\alpha}=\text{const}, \quad \quad
 \frac{\left|f\left(\Delta^{Q_s}_{(s-1)}\right)-f\left(\Delta^{Q_s}_{c_{n+1}\dots}\right)\right|}{\left|\Delta^{Q_s}_{(s-1)}-\Delta^{Q_s}_{c_{n+1}\dots}\right|^\alpha}\leq \frac{M}{q_{s-1}^\alpha}.
\end{equation*}
Therefore, if $\alpha\leq \frac{\ln |g_{s-1}|}{\ln q_{s-1}}$, then $\frac{|g_{s-1}|}{q^\alpha_{s-1}}\leq 1$ and $\limsup\limits_{x\to x_0-}\frac{\left|f(x_0)-f(x)\right|}{\left|x_0-x\right|^\alpha}<\infty$.
\end{proof}

\begin{theorem}
The function $f$ does not satisfy the H\"older condition at a $Q_s$-binary point $x_0$ with exponent $\alpha>K$.
\end{theorem}

\begin{proof}
Let $x_0 =\Delta^{Q_s}_{b_1 \ldots b_{k-1}b_k(0)}=\Delta^{Q_s}_{b_1\ldots b_{k-1}[b_k-1](s-1)}$, where $b_k\neq 0$. 
    
If $\alpha > \frac{\ln |g_0|}{\ln q_0}$, consider the $Q_s$-representation of $x_0$ such that it contains the period $(0)$, and the sequence $(x_n)^{\infty}_{n=1}$ given by
\begin{equation*}
x_n=\Delta^{Q_s}_{b_1 \ldots b_k \underbrace{\scriptstyle 0\ldots0}_{n} (s-1)}.
\end{equation*}
Then we obtain
\begin{equation*}
     \frac{|f(x_0)-f(x_n)|}{|x_0-x_n|^\alpha}=\left(\frac{|g_0|}{q_0^\alpha}\right)^n\cdot\prod^k_{i=1} \frac{|g_{b_i}|}{q_{b_i}^\alpha} \cdot \frac{\left|f\left(\Delta^{Q_s}_{(0)}\right)-f\left(\Delta^{Q_s}_{(s-1)}\right)\right|}{\left|\Delta^{Q_s}_{(0)}-\Delta^{Q_s}_{(s-1)}\right|^\alpha} = \left(\frac{|g_0|}{q_0^\alpha}\right)^n\cdot\prod^k_{i=1} \frac{|g_{b_i}|}{q_{b_i}^\alpha}.
\end{equation*}
From the condition $\alpha> \frac{\ln |g_0|}{\ln q_0}$, it follows that $\frac{|g_0|}{q_0^\alpha} >1$. Consequently,
\begin{equation*}
\lim\limits_{n\to\infty}\frac{|f(x_0)-f(x_n)|}{|x_0-x_n|^\alpha}=\infty \qquad \text{and} \qquad \limsup\limits_{x\to x_0}\frac{|f(x_0)-f(x)|}{|x_0-x|^\alpha}=\infty.
\end{equation*} 

If $\alpha > \frac{\ln |g_{s-1}|}{\ln q_{s-1}}$, consider the $Q_s$-representation of $x_0$ that contains the period $(s-1)$ and the sequence $(x'_n)^{\infty}_{n=1}$ given by
\begin{equation*}
x'_n=\Delta^{Q_s}_{b_1 \ldots b_k \underbrace{\scriptstyle [s-1]\ldots[s-1]}_{n} (0)}.
\end{equation*}
Analogously, we obtain
\begin{equation*}
     \frac{|f(x_0)-f(x'_n)|}{|x_0-x'_n|^\alpha} = \left(\frac{|g_{s-1}|}{q_{s-1}^\alpha}\right)^n\cdot\prod^k_{i=1} \frac{|g_{b_i}|}{q_{b_i}^\alpha}.
\end{equation*}
From the condition $\alpha > \frac{\ln |g_{s-1}|}{\ln q_{s-1}}$, it follows that $\frac{|g_{s-1}|}{q_{s-1}^\alpha} >1$. Consequently, 
\begin{equation*}
\lim\limits_{n\to\infty}\frac{|f(x_0)-f(x'_n)|}{|x_0-x'_n|^\alpha}=\infty \qquad \text{and} \qquad \limsup\limits_{x\to x_0}\frac{|f(x_0)-f(x)|}{|x_0-x|^\alpha}=\infty.
\end{equation*}
Thus, for $\alpha > K$ the H\"older condition at the point $x_0$ is not satisfied.
\end{proof}

\begin{corollary}
The local H\"older exponent $\alpha_f(x)$ of the function $f$ at a $Q_s$-binary point $x$ is equal to $K$.
\end{corollary}

\section{Global extrema of the function $f$}
In the case where the function $f$ is nowhere monotonic, finding its global extrema is a non-trivial task. In \cite{PrGLSv2022}, global extrema were found only in the simplest case, namely when $0<\delta_i<1$ for all $i\in A_s\setminus\{0\}$.

In this section, we consider the function $f$ satisfying the following additional conditions:
\begin{equation}\label{conditions}
    \begin{cases}
      g_k<0 \text{ and } \delta_k>1 \text{ for some }k\in A_s,\\
      g_i>0\text{ for all }i\in A_s\setminus\{k\}.
    \end{cases}
\end{equation}
Under these conditions, $k\geq2$, $\delta_i>0$ for all $i\in A_s\setminus\{0\}$, and the function $f$ is nowhere monotonic. Moreover,
$$M=\max\limits_{x\in[0,1]}f(x)\geq f\left(\Delta^{Q_s}_{k(0)}\right)=\delta_{k}>1.$$

\begin{lemma}\label{lemmamaxmin}
Let $M=\max\limits_{x\in[0,1]}f(x)$ and $m=\min\limits_{x\in[0,1]}f(x)$. Then
\begin{equation*}
\max_{x\in\Delta^{Q_s}_{i}}f(x)=\delta_i+g_iM,\quad \min_{x\in\Delta^{Q_s}_{i}}f(x)=\delta_i+g_im
\end{equation*}
for $i\in A_s\setminus\{k\}$ and
\begin{equation*}
\max_{x\in\Delta^{Q_s}_{k}}f(x)=\delta_{k}+g_{k}m, \quad \min_{x\in\Delta^{Q_s}_{k}}f(x)=\delta_{k}+g_{k}M.
\end{equation*}
\end{lemma}
The result follows directly from the equality \eqref{easy1} with $n=1$, since $g_{k}<0$ and $g_i>0$ for $i\in A_s\setminus\{k\}$.

\begin{lemma}\label{lemmamin}
$\displaystyle m=\min_{x\in[0,1]} f(x)=\min\left\{0,\delta_{k}+Mg_{k}\right\}$, and $|m|<M$.
\end{lemma}
\begin{proof}
Clearly,
\begin{equation*}
    m=\min_{i\in A_s}\left\{\min_{x\in\Delta^{Q_s}_i}f(x)\right\}\leq 0.
\end{equation*}
By Lemma \ref{lemmamaxmin}, if $m<0$, then
\begin{equation*}
 \min_{x\in\Delta^{Q_s}_{i}}f(x)=\delta_i+g_im >\delta_i+m \geq m
\end{equation*}
for every $i\in A_s\setminus\{k\}$. Therefore, $m=\min\limits_{x\in\Delta^{Q_s}_{k}}f(x)=\delta_{k}+g_{k}M$. At the same time,
\begin{equation*}
  |m|=-\delta_k-g_kM=-\delta_k+|g_k|M<M. \qedhere
\end{equation*}
\end{proof}

\begin{lemma}\label{lemmaMnot}
$M\not=\max\limits_{x\in\Delta^{Q_s}_{k}}f(x)$.
\end{lemma}
\begin{proof}
Assume, for the sake of contradiction, that $M=\max\limits_{x\in\Delta^{Q_s}_{k}}f(x)$. Then from Lemma \ref{lemmamaxmin} it follows that
\begin{equation*}
M=\delta_{k}+g_{k}m.
\end{equation*}
Using the Lemma \ref{lemmamin}, we obtain two cases:
\begin{equation*}
M=\delta_{k}\quad\text{or}\quad M=\delta_{k}+g_{k}\left(\delta_{k}+g_{k}M\right).
\end{equation*}
Since $\delta_{k}>1$, $k\geq2$ and $0<g_{k-1}<1$, we have
\begin{equation*}
    f\left(\Delta^{Q_s}_{[k-1]k(0)}\right)=\delta_{k-1}+\delta_{k}g_{k-1}>\delta_{k-1}+g_{k-1}=\delta_{k}.
\end{equation*}
Therefore, the case $M=\delta_{k}$ is impossible. In the second case, solving for $M$ from equation $M=\delta_{k}+g_{k}\left(\delta_{k}+g_{k}M\right)$, we obtain
\begin{equation*}
M=\frac{\delta_{k}}{1-g_{k}}=\frac{\delta_{k}}{1+|g_{k}|}<\delta_{k}.
\end{equation*}
Thus, this case is also impossible. Therefore, the assumption is false and $M\not=\max\limits_{x\in\Delta^{Q_s}_{k}}f(x)$.
\end{proof}

\begin{theorem}\label{theoremmax}
The maximum value of the function $f$ is given by
$$M=\max_{i\in A_s}\left\{\frac{\delta_i}{1-g_i}\right\}.$$
\end{theorem}
\begin{proof}
Using Lemmas \ref{lemmamaxmin} and \ref{lemmaMnot}, we have
\begin{equation*}
M=\max_{i\in A_s\setminus\{k\}}\left\{\max\limits_{x\in\Delta^{Q_s}_{i}}f(x)\right\}=\max_{i\in A_s\setminus\{k\}}\left\{\delta_i+g_iM\right\}.
\end{equation*}
Therefore, there exists $j\in A_s\setminus\{k\}$ such that $M=\delta_j+g_jM$, which gives $M=\frac{\delta_j}{1-g_j}$.

Let $t$ be a number given by
$\frac{\delta_t}{1-g_t}=\max\limits_{i\in A_s\setminus\{k\}}\left\{\frac{\delta_i}{1-g_i}\right\}$.
Then
$M=\frac{\delta_j}{1-g_j}\leq\frac{\delta_t}{1-g_t}.$
Since 
$$f\left(\Delta^{Q_s}_{(t)}\right)=\delta_t+\delta_t g_t+\delta_t g_t^2+\cdots=\frac{\delta_t}{1-g_t},$$
we have $M= \frac{\delta_t}{1-g_t}$. It is easy to check that
$\frac{\delta_{k-1}}{1-g_{k-1}}>\frac{\delta_k}{1-g_k}.$
Therefore,
\begin{equation*}
M=\max\limits_{i\in A_s\setminus\{k\}}\left\{\frac{\delta_i}{1-g_i}\right\}=\max\limits_{i\in A_s}\left\{\frac{\delta_i}{1-g_i}\right\}.\qedhere
\end{equation*}
\end{proof}

\begin{figure}[htbp]
\centering
\begin{subfigure}[t]{0.45\linewidth}
    \centering
    \includegraphics[width=\linewidth]{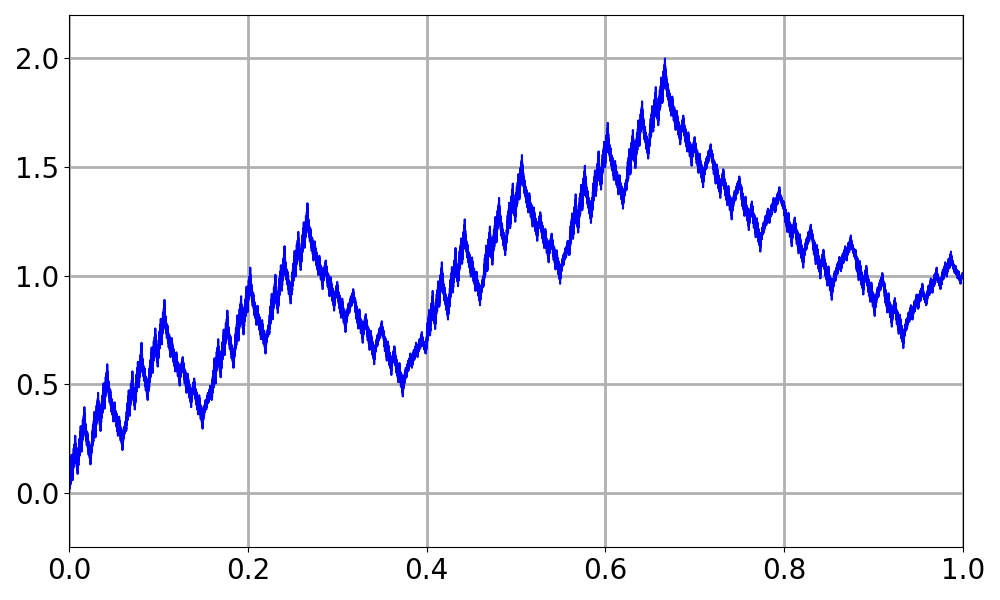}
    \caption{\centering Graph of the nowhere differentiable function $f$ for
    $q_0=q_1=\frac{2}{5}$, $q_2=\frac{1}{5}$ and
    $g_0=g_1=\frac{2}{3}$, $g_2=-\frac{1}{3}$.}
    \label{fig:f2}
\end{subfigure}
\hfill
\begin{subfigure}[t]{0.45\linewidth}
    \centering
    \includegraphics[width=\linewidth]{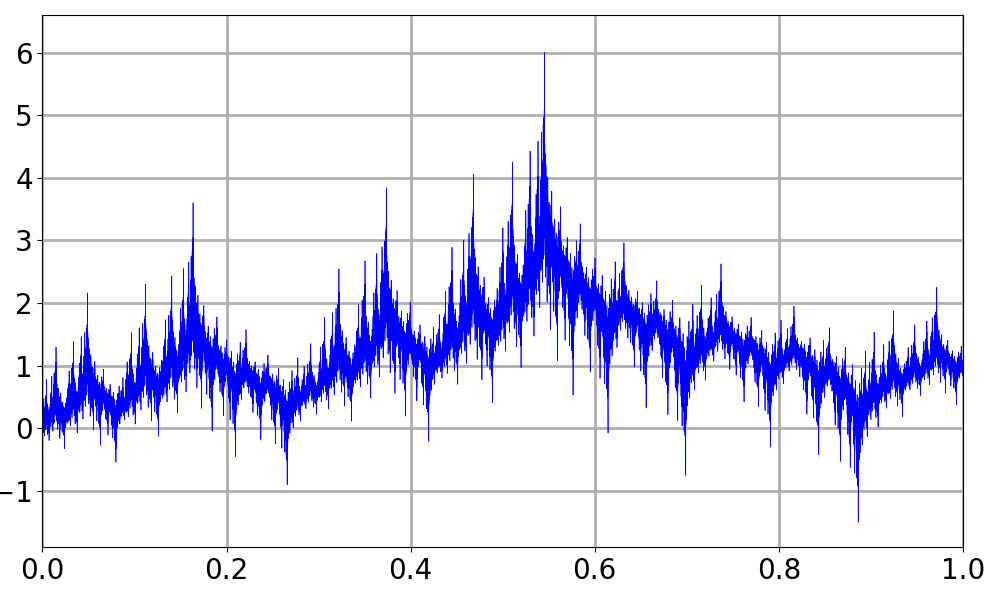}
    \caption{\centering Graph of the nowhere differentiable function $f$ for $q_0=\frac{3}{10}$, $q_1=\frac{9}{20}$, $q_2=\frac{1}{4}$ and $g_0=\frac{3}{5}$, $g_1=\frac{9}{10}$, $g_2=-\frac{1}{2}$ \newline with a negative minimum.}
    \label{fig:f3}
\end{subfigure}
\caption{Examples of graphs of the function $f$.}
\label{fig:two_examples}
\end{figure}

\begin{corollary}\label{cor1}
If $\frac{\delta_i}{1-g_i}<M$, then $\max\limits_{x\in\Delta^{Q_s}_{i}}f(x)<M$.
\end{corollary}
\begin{proof}
From the condition we have $\delta_i<(1-g_i)M$. If $i\not=k$, then
\begin{equation*}
\max\limits_{x\in\Delta^{Q_s}_{i}}f(x)=\delta_i+g_iM<(1-g_i)M+g_iM=M.
\end{equation*}
If $i=k$, then the statement follows directly from Lemma \ref{lemmaMnot}.
\end{proof}

\section{Continuum level sets and fractal structure of the set of maximum points of the function $f$}

\begin{definition}
The level set of value $y$ for a function $F\colon [0,1]\to \mathbb{R}$ is defined as the set
\begin{equation*}
    F^{-1}(y)=\left\{x\in [0,1]\colon F(x)=y\right\}.
\end{equation*}
\end{definition}

\begin{lemma}\label{levellemma}
If the set $V(y)$ given by
\begin{equation*}
  V(y)=\left\{i\in A_s\colon \frac{\delta_i}{1-g_i}=y\right\}  
\end{equation*}
is non-empty for some $y$, then
\begin{equation*}
f^{-1}(y)\supseteq C\left[Q_s,V(y)\right]=\left\{x=\Delta^{Q_s}_{\alpha_1 \alpha_2 \ldots}\colon \alpha_n\in V(y)\text{ for all }n\in\mathbb{N}\right\}.
\end{equation*}
In particular, if $|V(y)|\geq 2$, then the set $f^{-1}(y)$ has the cardinality of the continuum.
\end{lemma}

\begin{proof}
Let $x=\Delta^{Q_s}_{\alpha_1 \alpha_2 \ldots}\in C\left[Q_s,V(y)\right]$. Then $\delta_{\alpha_n}=(1-g_{\alpha_n})y$. Using equality \eqref{easy1} and mathematical induction, it is easy to show that for every $n\in\mathbb{N}$ the following equality holds
\begin{equation}\label{1}
f\left(\Delta^{Q_s}_{\alpha_1 \alpha_2 \ldots}\right)-y=\left(f\left(\Delta^{Q_s}_{\alpha_{n+1} \alpha_{n+2} \ldots}\right)-y\right)\cdot\prod_{i=1}^{n}g_{\alpha_i}.
\end{equation}
Since $\prod\limits_{i=1}^{n}g_{\alpha_i}\to0$ as $n\to\infty$ and the expression $f\left(\Delta^{Q_s}_{\alpha_{n+1} \alpha_{n+2} \ldots}\right)-y$ is bounded, the right-hand side of the~equa\-lity~\eqref{1} tends to 0. Therefore, $f(x)=y$ and $C\left[Q_s,V(y)\right]\subseteq f^{-1}(y)$.

If $|V(y)|\geq 2$, then the set $C\left[Q_s,V(y)\right]$ has the cardinality of the continuum. Hence, the set $f^{-1}(y)$ is also a~continuum.
\end{proof}

\begin{example}\label{ex2}
Consider the function $f$ defined for the $Q_5$-representation with $q_0=q_4=\frac{1}{20}$, $q_1=q_3=\frac{7}{20}$, $q_2=\frac{1}{5}$ and $g_0=0.5$, $g_1=0.2$, $g_2=0.1$, $g_3=-0.28$, $g_4=0.48$. In this case
\begin{equation*}
    \frac{\delta_1}{1-g_1}=\frac{\delta_3}{1-g_3}=0.625
\end{equation*}
and $V(0.625)=\{1,3\}$. Therefore, the level set corresponding to $y=0.625$ is continuum and contains the subset $C\left[Q_5,\{1,3\}\right]$. The function $f$ is also singular, since $\sum\limits_{i=0}^{4} q_i\ln|g_i|\approx -1.54$ and $\sum\limits_{i=0}^{4} q_i\ln q_i \approx -1.36$.
\end{example}

\begin{figure}[h]
\centering
\includegraphics[width=0.65\textwidth]{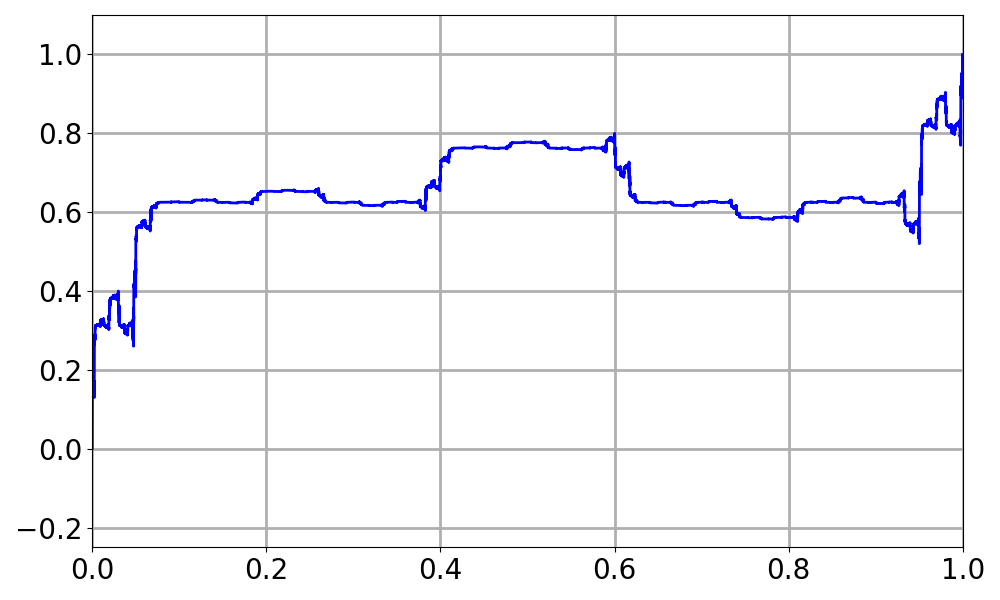}
\caption{\centering Graph of the nowhere monotonic singular function $f$ with continuum level sets from the~Example \ref{ex2}.}
\label{graphs3}
\end{figure}

\begin{corollary}
If $|V(y)|\geq 2$ for some $y$, then the function $f$ has at least countably many continuum level sets.
\end{corollary}
\begin{proof}
Let $y_n = g_0^n y$. Then $f(x)=y_n$ for all $x=\Delta^{Q_s}_{\underbrace{0\ldots0}_n\alpha_{n+1}\alpha_{n+2}\ldots}$, where $\alpha_k\in V(y)$ for all $k\geq n+1$. Clearly, the set of such numbers $x$ has the cardinality of the continuum.
\end{proof}

Now consider in more detail the case where the function $f$ satisfies conditions \eqref{conditions}.

\begin{theorem}
Let $V(M)=\left\{i\in A_s\colon \frac{\delta_i}{1-g_i}=M \right\}$, where $M=\max\limits_{x\in [0,1]}f(x)$. Then the~set of maximum points of the function $f$ is
    \begin{equation*}
        C\left[Q_s,V(M)\right]=\left\{x=\Delta^{Q_s}_{\alpha_1 \alpha_2\ldots}\colon \alpha_n\in V(M) \text{ for all } n\in\mathbb{N}\right\}.
    \end{equation*}
In particular,
\begin{enumerate}
    \item[1.] if $\left|V(M)\right|\geq 2$, then $C\left[Q_s,V(M)\right]$ is a continuum Cantor-type set with zero Lebesgue measure;
    \item[2.] if $V(M)=\{i\}$, then the function $f$ has a unique maximum point $x_{max}=\Delta^{Q_s}_{(i)}$.
\end{enumerate}
\end{theorem}
\begin{proof}
If $x\in C\left[Q_s,V(M)\right]$, then $C\left[Q_s,V(M)\right]\subseteq f^{-1}(M)$, which follows from Lemma \ref{levellemma}. We show that if $x\notin C\left[Q_s,V(M)\right]$, then $f(x)<M$.

Indeed, if $x=\Delta^{Q_s}_{\alpha_1 \alpha_2 \ldots}\notin C\left[Q_s,V(M)\right]$, then there exists an index $n$ such that $\alpha_{i}\in V(M)$ for all $i<n$ and $\alpha_n\notin V(M)$. Moreover, $\delta_{\alpha_i}=(1-g_{\alpha_i})M$ for all $i<n$ and $\delta_{\alpha_n}<(1-g_{\alpha_n})M$. Since
\begin{equation*}
  \frac{\delta_0}{1-g_0}=0<M,\qquad \frac{\delta_{k}}{1-g_{k}}<\delta_k<M,
\end{equation*}
it follows that $0\notin V(M)$ and $k\notin V(M)$. Therefore, $g_{\alpha_i}>0$ for $i<n$ and $\max\limits_{x\in\Delta^{Q_s}_{\alpha_n}}f(x)<M$ (by Corollary \ref{cor1}). Using the equality \eqref{1}, we obtain
\begin{align*}
f(x)-M=\left(f\left(\Delta^{Q_s}_{\alpha_{n} \alpha_{n+1} \ldots}\right)-M\right)\cdot\prod_{i=1}^{n-1}g_{\alpha_i}\leq\left(\max_{x\in\Delta^{Q_s}_{\alpha_n}}f(x)-M\right)\cdot\prod_{i=1}^{n-1}g_{\alpha_i}<0.
\end{align*}
Hence $f(x)<M$ and $f^{-1}(M)=C\left[Q_s,V(M)\right]$.

As shown in \cite{AKPrT2011}, if $2\leq |V(M)|\leq s-1$, then $C\left[Q_s,V(M)\right]$ is a continuum Cantor-type set with zero Lebesgue measure. If $V(M)=\{i\}$, then $C\left[Q_s,V(M)\right]=\left\{ \Delta^{Q_s}_{(i)}\right\}$.
\end{proof}

\begin{corollary}
If $s=3$, then $f$ has a unique maximum point $x_\text{max}=\Delta^{Q_3}_{(1)}$ and
\begin{equation*}
M=\max\limits_{x\in[0,1]}f(x)=\frac{\delta_1}{1-g_1}.
\end{equation*}
\end{corollary}

\begin{example}\label{ex1}
Let $f$ be a self-affine nowhere differentiable function defined by the parameters $q_0=q_2=q_3=\frac{1}{5}$, $q_1=\frac{2}{5}$ and $g_0 = g_2 = \frac{2}{5}$, $g_1= \frac{4}{5}$, $g_3=-\frac{3}{5}$. For this function
\begin{equation*}
  \max_{0\leq i\leq 3}\left\{\frac{\delta_i}{1-g_i}\right\}= \frac{\delta_1}{1-g_1} = \frac{\delta_2}{1-g_2}=2,
\end{equation*}
$M=2$ and $V(2)=\{1,2\}$. Then the set of maximum points of the function $f$ is the continuum nowhere dense set $C\left[Q_4,\{1,2\}\right]$ with zero Lebesgue measure. Moreover, its Hausdorff dimension $\dim_H C\left[Q_4,\{1,2\}\right]$ is the solution to the equation
\begin{equation*}
    \left(\frac{2}{5}\right)^x+\left(\frac{1}{5}\right)^x=1
\end{equation*}
and is approximately equal to $0.564$.

\begin{figure}[h]
\centering
\includegraphics[width=0.6\textwidth]{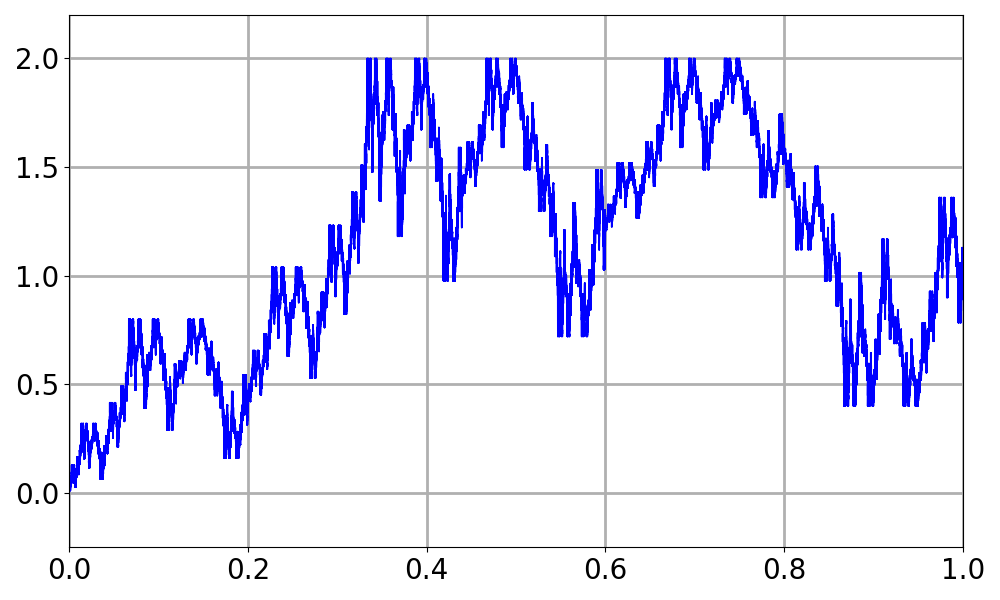}
\caption{\centering Graph of the nowhere differentiable function $f$ with a continuum set of maximum points from the Example \ref{ex1}.}
\label{graphs4}
\end{figure}

The set $C\left[Q_4,\{1,2\}\right]$ can be constructed geometrically (similar to the classical Cantor set). At the first step, the~interval $[0,1]$ is divided into four subintervals in the ratio $1:2:1:1$. The two central subintervals are kept, and the remaining points are deleted. At the second step, the same procedure is applied to the remaining two intervals. This process is continued inductively. The first 5 steps are illustrated in the Figure \ref{CantorTypeSet}. The set $C\left[Q_4,\{1,2\}\right]$ consists of those points that are never removed at any step.
\end{example}

\begin{figure}[h]
\centering
\includegraphics[width=0.7\textwidth]{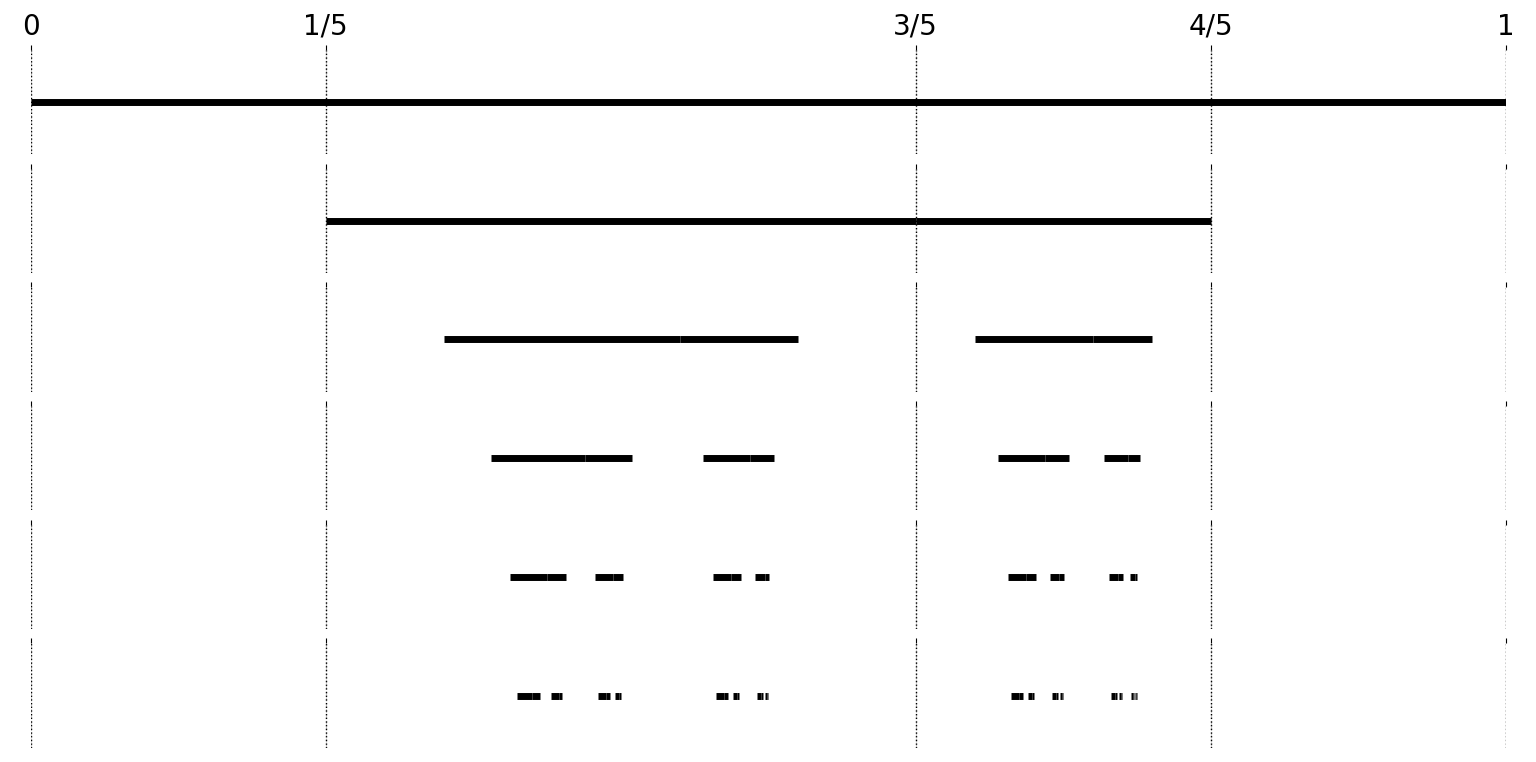}
\caption{Geometric construction of the set $C\left[Q_s,V\right]$ from the Example \ref{ex1}.}
\label{CantorTypeSet}
\end{figure}

\section{Non-invariance of metric, fractal, and topological characteristics of sets under the transformation $f$.}

Consider the case where the function $f$ satisfies conditions \eqref{conditions}.

Let $V^*=\left\{0,1,\ldots,k-1\right\}$. Consider the set
\begin{equation*}
C\left[Q_s,V^*\right]=\left\{x=\Delta^{Q_s}_{\alpha_1 \alpha_2\ldots}\colon \alpha_n\in V^* \text{ for all } n\in\mathbb{N}\right\}.
\end{equation*}

\begin{lemma}\label{lemmaint}
$[0,1]\subseteq f\left(C\left[Q_s,V^*\right]\right)$.
\end{lemma}

\begin{proof}
Consider an arbitrary number $y\in[0,1]$. Since $\delta_0=0$ and $\delta_{k}>1$, there exists a number $\alpha_1\in V^*$ such that $y\in[\delta_{\alpha_1},\delta_{\alpha_1+1}]$. Denote $y_1=\frac{y-\delta_{\alpha_1}}{g_{\alpha_1}}$. Then
\begin{gather*}
y=\delta_{\alpha_1}+g_{\alpha_1}y_1\qquad\text{and}\qquad 0\leq y_1\leq \frac{\delta_{\alpha_1+1}-\delta_{\alpha_1}}{g_{\alpha_1}}=1.
\end{gather*}
Similarly, there exists a number $\alpha_2\in V^*$ such that $y_1\in[\delta_{\alpha_2},\delta_{\alpha_2+1}]$. Denote $y_2=\frac{y_1-\delta_{\alpha_2}}{g_{\alpha_2}}$. Then
\begin{gather*}
    y_1=\delta_{\alpha_2}+g_{\alpha_2}y_2,\qquad  0\leq y_2\leq \frac{\delta_{\alpha_2+1}-\delta_{\alpha_2}}{g_{\alpha_2}}=1,\\
y=\delta_{\alpha_1}+\delta_{\alpha_2}g_{\alpha_1}+g_{\alpha_1}g_{\alpha_2}y_2.
\end{gather*}
Continuing these steps, we obtain infinite sequences $(y_n)_{n=1}^\infty$ and $(\alpha_n)_{n=1}^\infty$ such that 
\begin{equation*}
y=\delta_{\alpha_1}+\sum_{k=2}^n\left(\delta_{\alpha_k}\prod_{j=1}^{k-1}g_{\alpha_j}\right)+y_n\prod_{j=1}^{n}g_{\alpha_j}
\end{equation*}
for all $n\in\mathbb{N}$, where $y_i\in[0,1]$ and $\alpha_{i}\in V^*$ for all $i\in\mathbb{N}$. Then 
\begin{equation*}
    \left|y-f\left(\Delta^{Q_s}_{\alpha_1\alpha_2\ldots}\right)\right|=\prod_{j=1}^{n}g_{\alpha_j}\cdot \left|y_n-f\left(\Delta^{Q_s}_{\alpha_{n+1}\alpha_{n+2}\ldots}\right)\right|\leq (M-m)\cdot g_*^n
\end{equation*}
for each $n$, where $0<g_*=\max\limits_{i\in V^*}g_i<1$. Under these conditions, $g_*^n\to 0$ as $n\to\infty$. Therefore,
\begin{equation*}
    y=f\left(\Delta^{Q_s}_{\alpha_1\alpha_2\ldots}\right),
\end{equation*}
where $\Delta^{Q_s}_{\alpha_1\alpha_2\ldots}\in C[Q_s,V^*]$, whence $y\in f\left(C\left[Q_s,V^*\right]\right)$. Thus, $[0,1]\subseteq f\left(C\left[Q_s,V^*\right]\right)$.
\end{proof}

\begin{theorem}
The function $f$ does not preserve the Hausdorff dimension, Lebesgue null-measure, and the Baire topological category.
\end{theorem}

The proof of the theorem follows directly from the Lemma \ref{lemmaint}. Indeed, the set $C\left[Q_s,V^*\right]$ is of the first Baire category (since it is nowhere dense), has Lebesgue measure zero, and has Hausdorff dimension less than 1. At the same time, $f\left(C\left[Q_s,V^*\right]\right)$ is of the second Baire category (since it contains an interval), has positive Lebesgue measure (at least 1), and has Hausdorff dimension equal to 1.

Since the function $f$ satisfies the H\"older condition with a positive exponent $\alpha=\max\limits_{i\in A_s}\frac{\ln |g_i|}{\ln q_i}$, the inequality $\dim_H f(E)\leq \frac{1}{\alpha} \dim_H E$ holds for an arbitrary set $E\subset[0,1]$. Therefore, the function $f$ preserves Hausdorff dimension zero.

\section*{Author contributions}
The authors’ contributions are as follows: Mykola Moroz proposed the problem formulation and the research methodology and verified the results; Volodymyr Yelahin formulated and proved the main statements; both authors contributed equally to the writing and visualization. 

\section*{Funding}
This work was supported by a grant from the Simons Foundation (SFI-PD-Ukraine-00014586, V.O.Ye, M.P.M.)

\end{document}